\documentclass[11pt,a4paper]{article}
\usepackage {epsfig,amsmath,euscript}
\usepackage[latin1]{inputenc}
\usepackage[english]{babel}
\usepackage{color}
\usepackage{fancyheadings}
\usepackage{amsmath}
\usepackage{amssymb}
\usepackage[amsmath,thmmarks,standard,thref]{ntheorem}
\usepackage{ae}
\usepackage{subfigure}
\usepackage[T1]{fontenc}
\usepackage{graphicx}
\usepackage{epstopdf}
\usepackage{longtable}
\usepackage{color}
\definecolor{rred}{rgb}{0.7,0.0,0.2}
\definecolor{bblue}{rgb}{0.2,0.0,0.7}

\newcommand{\secref}[1]{section \ref{sec:#1}}

\newcommand{\seclab}[1]{\label{sec:#1}}

\newcommand{\eqlab}[1]{\label{eq:#1}}
\renewcommand{\eqref}[1]{(\ref{eq:#1})}
\newcommand{\eqsref}[2]{(\ref{eq:#1}) and~(\ref{eq:#2})}

\newcommand{\figref}[1]{Fig.~\ref{fig:#1}}
\newcommand{\figlab}[1]{\label{fig:#1}}

\newcommand{\lemmaref}[1]{Lemma~\ref{lemma:#1}}
\newcommand{\lemmalab}[1]{\label{lemma:#1}}
\newcommand{\remref}[1]{Remark~\ref{remark:#1}}
\newcommand{\remlab}[1]{\label{remark:#1}}
\newcommand{\thmref}[1]{Theorem~\ref{theorem:#1}}
\newcommand{\thmlab}[1]{\label{theorem:#1}}

\newcommand{\algoref}[1]{Algorithm~\ref{algorithm:#1}}
\newcommand{\algolab}[1]{\label{algorithm:#1}}
\newtheorem{algo}{Algorithm}

\newcommand{\R}{\mathbb R}
\usepackage{tikz}

\title {On the approximation of the canard explosion point\\ in epsilon-free systems}

\author{M. Br{\o}ns and K. Uldall Kristiansen\\
  Department of Applied Mathematics and Computer Science \\
  Technical University of Denmark \\
  2800 Kongens Lyngby \\
  Denmark } 
\begin{document}
\maketitle 

\begin{abstract}
  A canard explosion is the dramatic change of period and amplitude of
  a limit cycle of a system of non-linear ODEs in a very narrow
  interval of the bifurcation parameter. It occurs in slow-fast
  systems and is well understood in singular perturbation problems
  where a small parameter epsilon defines the time scale
  separation. We present an iterative algorithm for the determination
  of the canard explosion point which can be applied for a general
  slow-fast system without an explicit small parameter. We also present assumptions under which the algorithm gives accurate estimates of the canard explosion point.
  Finally, we apply the algorithm to the van der Pol equations
  and a Templator model for a self-replicating system with no explicit
  small parameter and obtain very good agreement with results from
  numerical simulations.
\end{abstract}


\pagestyle{myheadings}
\thispagestyle{plain}

\section{Introduction}
\label{sec:introduction}
We consider singular perturbation problems of the form
\begin{align}
  \dot x &=F(x,y,z,\epsilon),\eqlab{xydotGeneral}\\
  \dot y &=G(x,y,z,\epsilon),\nonumber
\end{align}
where $\epsilon$ is small parameter and $z$ is a bifurcation
parameter. All vector fields are assumed to be analytic in their
arguments. In the singular limit $\epsilon=0$ the system is assumed to
have a critical manifold, that is, a manifold of critical points. We
restrict our attention to planar systems with $x\in \R$ and $y\in \R$
such that the critical manifold is a curve. The parameter $\epsilon$
defines a time-scale separation, and for singular perturbation problems
in the special form 
\begin{align}
  \dot x &=F(x,y,z,\epsilon),\eqlab{xydot}\\
  \dot y &=\epsilon H(x,y,z,\epsilon),\nonumber
\end{align}
the state variables $x$ and $y$ can be identified as the fast and the
slow variable, respectively. For the system \eqref{xydot} the critical
manifold is defined by $F(x,y,z,0)=0$, and it follows from Fenichel
theory 
\cite{Fenichel1979:geometric-singular-perturbation,Jones1995:geometric-singular-perturbation}
that, under certain regularity conditions, there exists a slow
invariant manifold $\epsilon$-close to the critical manifold. For
simplicity, we will suppress the explicit dependence of $F$ and $H$ on
$\epsilon$ in the system \eqref{xydot} from now on.

\textbf{Canard explosion}. The canard explosion is the dramatic change
of amplitude and period of a limit cycle born in a Hopf bifurcation in
a very narrow parameter interval. The phenomenon is well understood in
singular perturbation problems of the form \eqref{xydot}
\cite{ben1,zvonkin-1984:non-standard,krupa-2001:relaxation,brons-2005:relaxation}. Canards
are solutions which are $\epsilon$-close to a critical manifold on the
form $\eta_0=\eta_0(y,z)$ defined from $F(\eta_0(y,z),y,z)=0$ both on
an attracting part where $\partial_xF < 0$ and a repelling part where
$\partial_xF > 0$. Thus, canards are slow manifolds with both an
attracting and a repelling part. They only exist for a narrow range
(of order $e^{-c/\epsilon}$) of the parameter $z$, and the canard
explosion occurs when a segment of a limit cycle is a canard. A unique
asymptotic expansion of a parameter value $z$ where the canard occurs,
the canard point, and an expansion of the corresponding (maximal)
canard can readily be obtained.

However, canard explosions are also observed in slow-fast systems
where there is no explicit small parameter that separates the
timescales,
\begin{align}
  \dot x &=F(x,y,z),\eqlab{epsfree}\\
  \dot y &=H(x,y,z),\nonumber
\end{align}
and where existing singular perturbation theory cannot be applied
directly. An example is the Templator model which we consider in
Section \ref{sec:applications}.  In
\cite{brons-2013:iterative} a modification of the iterative method of
Fraser and Roussel \cite{fra1,fra2}, devised to construct regular slow
manifolds, was proposed to determine a canard point for a general
system of the form \eqref{epsfree}. The method was successfully
applied to the van der Pol equations where it was shown that the first
terms in the asymptotic expansion of the canard point are determined
correctly, and on a Templator model for a self-replicating system with
no explicit $\epsilon$ where canard explosions have been found
numerically.

In this paper, we will consider a modified version of this
method. Being based on a linearization, the method is explicit, in
contrast to the original one. Furthermore, we present explicit non-degeneracy
conditions as well as ``smallness conditions'' and show how they
guarantee accurate approximations for the canard solution and the
canard explosion point for the system \eqref{epsfree}.  We illustrate
the method with examples.

\textbf{Notation}. All norms will be denoted by $\vert \cdot \vert$. Superscripts with $n\in \mathbb N_0$ will be used to denote partial sums such as:
\begin{eqnarray}
 \eta^n = \sum_{i=0}^n \eta_i,\quad n\ge 0,\eqlab{etaNN}
\end{eqnarray}
with each of the terms in the sum being enumerated through subscripts. Following this convention means that $\eta^0=\eta_0$.
We will in this paper suppose that the vector-field is analytic. In particular we will consider complex sets of the form $x\in (a,b)+i\nu$ and $y\in (c,d)+i\sigma$. Here we by $(a,b)+i\nu\subset \mathbb C$ and $(c,d)+i\sigma\subset \mathbb C$ denote the complex $\nu$ and $\sigma$-neighborhoods of the real intervals $(a,b)$ resp. $(c,d)$. We will then use $\vert h \vert_{\nu}$ to denote the sup-norm of an analytic function $h$ over the domain $(a,b)+i\nu$. This representation gives the following compact form of Cauchy's estimate:
 \begin{lemma}
  Let $h=h(x)$ be analytic in $x\in (a,b)+i\nu$. Then 
  \begin{align*}
   \vert h'\vert_{\nu-\xi}\le \frac{\vert h\vert_{\nu}}{\xi},\quad 0<\xi\le \nu.
  \end{align*}
 \end{lemma}
 \begin{proof}
  See \cite[Theorem 10.26]{Rudin:1987:RCA:26851}.
 \end{proof}

\textbf{The iterative method of Fraser and Roussel}. A slow
manifold for Eqns.~(\ref{eq:xydot}) of the form $x=\eta(y,z)$ fulfills
the invariance equation 
\begin{align}
 -\epsilon \partial_y \eta H(\eta,y,z) + F(\eta,y,z) = 0,\eqlab{inveqn}
\end{align}
obtained by eliminating time. Fraser and Roussel \cite{fra1,fra2} developed an iterative method for
the approximation of slow manifolds from the scheme
\begin{align}
 -\epsilon \partial_y \eta^{n-1} H(\eta^{n},y,z) + F(\eta^{n},y,z) = 0,\eqlab{standard_fr}
\end{align}
starting from the critical manifold $\eta^0=\eta_0(y,z)$. The function $\eta=\eta^{n}$ satisfies \eqref{inveqn} up to the error
\begin{align}
 -\epsilon \partial_y \eta_n H(\eta^{n},y,z),\eqlab{err_fr}
\end{align}
using the notation \eqref{etaNN}. 
In \cite{kri3} it was shown using Cauchy estimates to control
\eqref{err_fr} that this procedure, even for an arbitrary number of
slow and fast variables, leads to slow manifolds exponentially close
($\mathcal O(e^{-c/\epsilon})$) to invariance. The equation
\eqref{standard_fr} is non-linear in $\eta^{n}$, but Neishtadt
\cite{nei87} showed that the same convergence is obtained if $F$ and
$H$ in \eqref{standard_fr} are linearized with respect to $x$ at $\eta^{n-1}$ such that
the equation becomes linear in $\eta_n = \eta^n-\eta^{n-1}$.

A crucial assumption for the success of these methods is that
$\vert \partial_x F\vert \gg \epsilon$. This is violated
near fold points of the critical manifold where
$\partial_x F=0$, and this is where canards may
occur. To study canards in systems with fold points one
instead makes use of the fact that  $\partial_y F\ne 0$ and
solve $F(x,y,z)= 0$ for $y=\zeta_0(x,z)$. A fold point $(x_0,\zeta_0(x_0,\mu_0))$ is
then characterized by $\partial_x \zeta_0(x_0,\mu_0)=0$ for some
parameter value $z=\mu_0$. For a slow manifold $y=\zeta(x,z)$ the
invariance equation becomes
\begin{equation}
  \label{eq:1}
  -\partial_x\zeta F(x,\zeta,z) + \epsilon H(x,\zeta,z) = 0,
\end{equation}
and following Fraser and Roussel it is proposed in
\cite{brons-2013:iterative} to solve this iteratively from
\begin{equation}
  \label{eq:2}
-\partial_x\zeta^{n-1}F(x,\zeta^n,z) + \epsilon H(x,\zeta^n,z) = 0
\end{equation}
starting with the critical manifold $\zeta^0=\zeta_0$. A solution
$\zeta^n(x,z)$ of (\ref{eq:2}) will generally have a singularity close
to the fold, but this may be canceled by an appropriate choice of
$z=\mu^n$. Then $\zeta^n$ is well-defined at the fold, and hence
represents a canard. Thus, this procedure yields a sequence $\mu^n$ of
approximations to the canard point as well as approximations
$y=\zeta^n(x,\mu^n)$ of the corresponding canard.  In the present
paper we follow Neishtadt again, and linearize Eqn.~(\ref{eq:2}) at
$\zeta^{n-1}$ such that it can be solved explicitly for $\zeta^n$.  A
precise formulation of this procedure requires some notation which we
turn to now; the complete algorithm is described as Algorithm 1 below.


\section{The modified Fraser-Roussel algorithm for canards}
\label{sec:modif-fras-rouss}
Consider the system \eqref{xydotGeneral} and let
\begin{align*}
 V(x,y,z,\epsilon) = \begin{pmatrix}
                      F(x,y,z,\epsilon)\\
                      G(x,y,z,\epsilon)
                     \end{pmatrix},
\end{align*}
denote the vector-field.  A critical manifold for \eqref{xydotGeneral}
is a smooth curve of fixed points $V(x,y,z,0)=0$ within the
$(x,y)$-plane for \eqref{xydotGeneral}$_{\epsilon=0}$.  We shall
assume that this manifold can be parametrized by $x$ so that
$y=\zeta_0(x,z)$. We then have
\begin{lemma}\lemmalab{foldPoint}
 The critical manifold is normally hyperbolic at $(x_0,\zeta_0(x_0,z))$ if and only if
 \begin{align*}
  \partial_y V \cdot (-\partial_x \zeta_0,1)\ne 0.
 \end{align*}
\end{lemma}
\begin{proof}
 The layer equations of \eqref{xydotGeneral} are
 \begin{align*}
 \begin{pmatrix}
  \dot x\\
  \dot y
 \end{pmatrix} &=V(x,y,z,0),
 \end{align*}
 where $\dot x$ and $\dot y$ in general are both non-zero. 
 An equilibrium $(x_0,\zeta_0(x_0,z))$ is then normally hyperbolic if the linearization, described by the Jacobian:
 \begin{align*}
  DV = \begin{pmatrix}
\partial_x V & \partial_y V   
  \end{pmatrix},
 \end{align*}
 only has one zero eigenvalue. Since $V(x,\zeta_0(x,z),z,0)=0$ we have by implicit differentiation that 
 \begin{align*}
  \partial_x V = -\partial_y V\partial_x \zeta_0.
 \end{align*}
Therefore we can write $DV$ as
\begin{align*}
 DV = \begin{pmatrix}
       -\partial_y V\partial_x \zeta_0 & \partial_y V
      \end{pmatrix}.
\end{align*}
The eigenvalues of $DV$ are then directly obtained:
\begin{align*}
 0 \quad \text{and}\quad \partial_y V \cdot (-\partial_x \zeta_0,1),
\end{align*}
and the result therefore follows.
\end{proof}

Following \lemmaref{foldPoint}, the critical manifold $y=\zeta_0(x,z)$ loses normal hyperbolicity at point $(x_0,\zeta_0(x_0,z))$ where
\begin{align}
  \partial_y V \cdot (-\partial_x \zeta_0,1)= 0. \eqlab{foldPoint}
\end{align}
Generically (see (A1) below) such a point is a fold point. 
\begin{remark}
 Geometrically, \eqref{foldPoint} means that the critical fiber of $(x_0,\zeta_0(x_0,z))$ is tangent to the graph $y=\zeta_0(x,z)$ at this point. 
\end{remark}

One of the main aims of this paper is to present a method that applies
to systems of the form \eqref{epsfree} without an explicit
$\epsilon$. To this end, we define a manifold $y=\zeta_0(x,z)$ from
\begin{align*}
 F(x,\zeta_0(x,z),z)=0.
\end{align*}
Were there an $\epsilon$ multiplying $H$, as in \eqref{xydot}, this
would be the critical manifold for $\epsilon=0$ which can be used be
used as a starting point for a Fraser-Roussel iteration as described
in \S~\ref{sec:introduction}.  The precise assumptions needed for this
to work in the $\epsilon$-free setting are given in (A1) and (A2)
below.


To continue, let $(x_0,\mu_0)$, be fixed, their values to be
determined later in (A1), and introduce $y_0$ and $z_0$ by
\begin{align*}
y&=\zeta_0(x,z) + y_0,\\
z&=\mu_0+z_0.
\end{align*}
This leads to the following extended system obtained from \eqref{xydot}:
\begin{align}
 \dot x & =F_0(x,y_0,z_0)=f_0(x)+F_{0y}(x,y_0,z_0) y_0+F_{0z}(x,y_0,z_0)z_0,\eqlab{xy0z0eqn}\\
 \dot y_0&= e_0(x)+(\Phi(x)+b_0(x))z_0+T_0(x,z_0)+(\Lambda(x)+a_0(x,z_0))y_0+R_0(x,y_0,z_0),\nonumber\\
 \dot z_0&=0,\nonumber
\end{align}
where
\begin{align}
F_{0y}(x,y_0,z_0) &= (\text{using that $F(x,\zeta_0(x,\mu_0),\mu_0)= 0$})\nonumber\\
&=\int_0^1 \partial_y F(x,\zeta_0(x,\mu_0)+sy_0,\mu_0+s z_0)ds,\nonumber\\
F_{0z}(x,y_0,z_0)&=\int_0^1 \partial_z F(x,\zeta_0(x,\mu_0)+s y_0,\mu_0+s z_0)ds,\nonumber\\
e_{0}(x) &=G(x,\zeta_0(x,\mu_0),\mu_0),\eqlab{e0}\\
\Phi(x)&= -\partial_x \zeta_{0} (x,\mu_0) \partial_z F(x,\zeta_0(x,\mu_0),\mu_0) +\partial_z G(x,\zeta_{0}(x,\mu_0),\mu_0),\eqlab{B0}\\
\Lambda (x)&=-\partial_x \zeta_0 \partial_y F(x,\zeta_0(x,\mu_0),\mu_0)+ \partial_y G(x,\zeta_0(x,\mu_0),\mu_0),\eqlab{A0eqn}
\end{align}
and $T_0=\mathcal O(z_0^2)$, $a_0=\mathcal O(z_0)$, and $R_0=\mathcal O(y_0^2)$. Here we have just Taylor-expanded the right hands sides about $(y_0,z_0)=(0,0)$. For later convenience we have also introduced $f_0\equiv 0$ and $b_0\equiv 0$. The subscripts on $y_0$, $z_0$ and the functions $f_0,\ldots, b_0,\ldots,R_0$ are used to indicate that they later will be part of an iteration. The functions $\Lambda$ and $\Phi$ will not be updated.

The function $\Lambda=\Lambda(x)$ in \eqref{A0eqn} is precisely
$\partial_y V \cdot (-\partial_x \zeta_0,1)$ which according to \eqref{foldPoint} vanishes at a fold point $(x_0,\zeta_0(x_0,\mu_0))$ at a given parameter value $z=\mu_0$. Therefore we will assume that there is such a point where $\Lambda$ vanishes. This only gives one condition on the pair $(x_0,\mu_0)$. For the possibility of having a canard solution we need further conditions. These are contained in the following assumptions:
\begin{itemize}
 \item[(A1)] \textit{The pair $(x_0,\mu_0)$ is so that $\Lambda$ \eqref{A0eqn} and $e_0$ \eqref{e0} vanish:
\begin{align*} 
 -\partial_x \zeta_0(x_0,\mu_0) \partial_y F(x,\zeta_0(x_0, \mu_0),\mu_0)+&\partial_y G(x_0,\zeta_0(x,\mu_0),\mu_0)= 0,\\ 
 G(x_0,\zeta_0(x_0,\mu_0),\mu_0)&=0.
  \end{align*}}
  \end{itemize}
  Given that $\Lambda(x_0)=0$ and $e_0(x_0)=0$ we can define $\tilde \Lambda$ and $\tilde e_0$ by 
  \begin{align}
  \tilde \Lambda(x) &= \int_0^1 \partial_x \Lambda(x_0+s(x-x_0))ds,\nonumber\\
  &\text{respectively}\nonumber\\
  \tilde e_0(x) &= \int_0^1 \partial_x e_0(x_0+s(x-x_0))ds, \label{eq:tildee0}
  \end{align}
  so that $\Lambda(x) = (x-x_0)\tilde \Lambda(x)$ and $e_0(x) = (x-x_0)\tilde e_0(x)$.  We further assume: 
  \begin{itemize}
    \item[(A2)] \textit{The following non-degeneracy and ``smallness'' conditions hold true: Let 
    \begin{align}
     \tilde \delta_0 &\equiv \vert \tilde e_0\vert_{\nu_0},\eqlab{tdelta0}\\
     K &\equiv \vert \tilde \Lambda^{-1}\vert_{\nu_0}.\nonumber
    \end{align}
    Then there exist $\epsilon\ll 1$  so that 
    \begin{align}
    \tilde \delta_0 &\le K^{-1}\epsilon,\eqlab{K0}
    \end{align}
    and 
    \begin{align*}
     \vert F_{0y}\vert_{\nu_0},\,\vert F_{0z}\vert_{\nu_0},\, \vert T_0\vert_{\nu_0},\,\vert R_0\vert_{\nu_0}\ll K^{-1}\epsilon^{-1}.
    \end{align*}
    Furthermore, either of the following conditions hold true:
       \begin{itemize}
       \item[(a)]
    \begin{align}
    \vert \Phi(x_0) \vert^{-1} \ll K^{-1}\epsilon^{-1},\eqlab{B0est1a}
    \end{align}
    and 
    \begin{align}
     \vert \Phi\vert_{\nu_0} \ll K^{-1} \epsilon^{-1}. \eqlab{B0est2a}
    \end{align}
    \item[(b)] 
    \begin{align}
    \vert \Phi(x_0) \vert^{-1} &\le K^{-1}\epsilon^{-1},\eqlab{B0est1b}
    \end{align}
    and 
    \begin{align}
     \vert \Phi\vert_{\nu_0} \le K^{-1} \epsilon. \eqlab{B0est2b}
    \end{align}
     \end{itemize}
}

\end{itemize}
 \begin{remark}
 We include case (b) in (A2) to cover the case where $G=\epsilon H$ is small. 
 \end{remark}


The reason for supposing analyticity is that we in this case can present what we believe are optimal \textit{exponential estimates}. However, we will not need any smoothness assumptions on how $\epsilon$ enters beyond condition (A1) and (A2). In particular, the analysis of our method is not based on asymptotic expansions in $\epsilon$. 


The function $\tilde e_0$ in \eqref{tildee0} is the \textit{obstacle to invariance} of $y_0=0,\,z_0=0$: If $\tilde e_0\equiv 0$ then $y_0=0,z_0=0$ corresponds to a canard solution. We have therefore introduced $\tilde \delta_0$ in \eqref{tdelta0} as the error $\tilde \delta_0=\vert \tilde e_0\vert_{\nu_0}=\mathcal O(\epsilon)$. The system \eqref{xy0z0eqn} is our normal form. Our algorithm will be based upon applying transformations, affine in $y_0$ and $z_0$, to \eqref{xy0z0eqn} that seek to diminish the \textit{error} $\tilde e_0$. These transformations directly lead to a simple algorithm similar to that presented in  \cite{brons-2013:iterative} that we present in \algoref{algo}. 



%
\begin{remark}
Suppose that $G=\epsilon H$ is truly small and the slow and fast variables have been properly identified. 
Then there are known sufficient conditions for a canard explosion \cite{brons-2005:relaxation,kru2}:
\begin{itemize}
 \item[(B1)] There exists a pair $(\tilde x_0,\tilde \mu_0)$ so that 
 \begin{align*}
  \partial_x \zeta_0(\tilde x_0,\tilde \mu_0) = 0,\quad H(\tilde x_0,\zeta_0(\tilde x_0,\tilde \mu_0),\tilde \mu_0) = 0.
 \end{align*}
\item[(B2)] The following non-degeneracy conditions hold true:
\begin{align*}
 \partial_y F(\tilde x_0,\zeta_0(\tilde x_0,\tilde \mu_0),\tilde \mu_0) \ne 0,\,&\partial_x^2 F(\tilde x_0,\zeta_0(\tilde x_0,\tilde \mu_0),\tilde \mu_0)\ne 0,\\
 \partial_z H(\tilde x_0,\zeta(\tilde x_0,\tilde \mu_0),\tilde \mu_0) \ne 0,\,&\partial_x H(\tilde x_0,\zeta_0(\tilde x_0,\tilde \mu_0),\tilde \mu_0) \ne 0,
\end{align*}
\end{itemize}
Our condition (A1) replaces (B1). They are equivalent when
$\partial_x^2 F(x_0,\zeta_0(x_0,\mu_0),\mu_0)\ne 0$ by the implicit
function theorem. Similarly, using (A1), it follows that the first
three inequalities (B2) are equivalent to those in (A2), case (b), for
$\epsilon$ sufficiently small.  In condition (A2), however, we do not
require
$\partial_x H(\tilde x_0,\zeta_0(\tilde x_0,\tilde \mu_0),\tilde
\mu_0) \ne 0$.
This condition is included in (B2) because it guarantees that the
nullclines of $x$ and $y$ are transverse at the fold point and that
the equilibrium undergoes a Hopf bifurcation. It is the associated
limit cycles that undergo rapid amplitude growth in the canard
explosion. In agreement with \cite{brons-2005:relaxation}, our
algorithm and main result (\thmref{expest} below) still apply without
the need of this assumption but the results may have little dynamical
significance. There is a well-known connection between the first Liapounov
coefficient, the Hopf point and the canard point to lowest order in
$\epsilon$ \cite{krupa-2001:relaxation,brons-2005:relaxation}. This was exploited in
\cite{Kuehn2010:first-lyapunov-coefficients} as a numerical tool to estimate
canard explosion points.
\end{remark}
\begin{remark}\remlab{rema0}
 Suppose again that the slow and fast variables have been properly identified. Then one can actually replace $\Lambda$ in \eqref{A0eqn} by $-\partial_x \zeta_0 \partial_y F$ and ignore the small term $\partial_y (\epsilon H)$. Our main result still applies as the term ignored can be collected into $a_0$. Indeed, the iterative lemma, \lemmaref{itlam}, that is the basis of our proof of the main theorem, just assumes that $a_0(x,0)$ is bounded from above by $c\epsilon$, see \eqref{best} below, for some $c$ sufficiently large.  However, in the general case, this term cannot be ignored. See also \secref{templator2} where we apply our algorithm to the Templator model. 
\end{remark}
\begin{remark}
 As described in \remref{rema0}, the result of the paper still applies if $\vert a_0(x,0)\vert_{\nu_0} \ll \vert \tilde \Lambda\vert_{\nu_0}$. Similarly, we can also allow for $\vert f_0\vert_{\nu_0}\ll\vert \tilde \Lambda\vert_{\nu_0}$ and $\vert b_0\vert_{\nu_0}\ll\vert \Phi\vert_{\nu_0}$. The iterative lemma, \lemmaref{itlam}, still applies (see also \eqref{best} below).
\end{remark}

\textbf{The modified iterative method for the computation of canard explosion}. We are looking for a canard solution through a graph $y=\zeta(x)$ for a value $\mu$ of the parameter $z$. The invariance of the graph gives the equation (\ref{eq:1})
for $\zeta=\zeta(x)$, which following Neishtadt \cite{nei87} we wish to approach iteratively starting from $y=\zeta_0(x,\mu_0)$, $z=\mu_0$ and continue with solving the linear equations:
\begin{align}
\rho_{n-1}(x,z)&+\Lambda(x) \zeta_n(x,z)=0 ,\eqlab{update}\\
\rho_{n-1}(x,z) &=- \partial_x \zeta^{n-1} F(x,\zeta^{n-1},z)+G(x,\zeta^{n-1},z),\eqlab{rhon_1eqn1}\\
\Lambda(x) &=-\partial_x \zeta_{0}  \partial_y F(x,\zeta_{\text{0}},\mu_0)+\partial_y G(x,\zeta_{0},\mu_0).\eqlab{A0eqn1}
\end{align}
Due to our assumption (A1), however, we are potentially dividing by zero as $\Lambda(x_0)=0$ when solving for the updates $\zeta_{n}$. 
The idea of \cite{brons-2013:iterative} is then to proceed by removing the singularity in the expression for $\zeta_n$ by solving the equation $\rho_{n-1}(x_0,z)=0$ for $z=\mu^n$ so that $-\rho_{n-1}(x,\mu^n)/\Lambda(x)$ is well-defined. We collect this into an algorithm in the following:

\begin{algo}\algolab{algo}
 Suppose (A1) and (A2). To compute the canard explosion point do the following: 
 \begin{itemize}
 \item[$1^\circ$] Define $\tilde \Lambda=\tilde \Lambda(x)$ to be
 \begin{align*}
  \tilde \Lambda(x) = \int_0^1 \partial_x \Lambda (x_0+s(x-x_0),\mu_0)ds.
 \end{align*}
 Then by construction $\Lambda(x)=(x-x_0)\tilde \Lambda(x)$ where $\Lambda$ is defined in \eqref{A0eqn1}. 
 \item[$2^\circ$] Define
 \begin{align*}
 \tilde e_{0}(x) = \int_0^1 \partial_x \rho_{0}(x_0+s(x-x_0),\mu_0)) ds.
\end{align*}
Then by construction $\rho_{0}(x,\mu_0) = (x-x_0)\tilde e_{0}(x)$ where $\rho_0$ is defined in \eqref{rhon_1eqn1}$_{n=1}$.
 \item[$2^\circ$] Iterate the following over $n$ starting from $n=1$ and $\zeta^{0}=\zeta_0$ and $\mu^{0}=\mu_0$ until $\vert \tilde e_n\vert$ has reached a desired tolerance: 
  \begin{itemize}
 \item[(i)]
 Define
\begin{align*}
\zeta_n(x) &=  -\frac{\tilde e_{n-1}(x)}{\tilde \Lambda(x)},\quad
 \zeta^n(x) = \zeta^{n-1}(x)+\zeta_n(x).
\end{align*}
\item[(ii)]  Solve the following equation for $z=\mu_n$:
\begin{align*}
 \rho_{n}(x_0,\mu^{n-1}+z) &=0,
\end{align*}
where
\begin{align*}
\rho_{n}(x,\mu^{n-1}+z)=- \partial_x \zeta^{n} F(x,\zeta^{n},\mu^{n-1}+z)+G(x,\zeta^{n},\mu^{n-1}+z). 
\end{align*}
\item[(iii)]
Set $$\mu^n=\mu^{n-1}+\mu_n,$$ and let
\begin{align}
 \tilde e_{n}(x) = \int_0^1 \partial_x \rho_{n}(x_0+s(x-x_0),\mu^n)) ds.\eqlab{tildeen}
\end{align}
Then by construction $\rho_{n}(x,\mu^n) = (x-x_0)\tilde e_{n}(x)$.
\end{itemize}
\end{itemize}
The graph $y=\zeta^n(x)$ is then the approximation of the canard slow manifold, connecting repelling and attracting branches, at the explosion point $z=\mu^n$. The error is described in \thmref{expest}.
\end{algo}

\textbf{Main result}. Our main result is contained in the following theorem which we prove in \secref{proof}.
\begin{theorem}\thmlab{expest}
Suppose that the assumptions (A1) and (A2) hold true and that $F$ and $G$ are analytic in their arguments $(x,y,z)$. (i) Fix first $n\ge 0$. Then, provided $\epsilon$ is sufficiently small, the procedure defined in \algoref{algo} generates a sequence of $\zeta_i$'s and $\mu_i$'s so that 
 \begin{align*}
  \zeta^n(x) = \sum_{i=0}^n \zeta_i(x),\quad \mu=\mu^n = \sum_{i=0}^n \mu_i
 \end{align*}
satisfies (\ref{eq:1}) up to an error \eqref{tildeen} of $\tilde e_n=\mathcal O(\epsilon^{n+1})$. Moreover, (ii) there exists an $N(\epsilon)=\mathcal O(\epsilon^{-1/2})\in \mathbb N$ so that the error in \eqref{tildeen} with $n=N$ is exponentially small $\tilde e_N =\mathcal O(e^{-c/\epsilon^{1/2}})$, with $c>0$ and independent of $\epsilon$.
\end{theorem}
The first part (i) only requires smoothness. The last part (ii) requires analyticity.
\begin{remark}
 The $\sqrt{\epsilon}$ in \thmref{expest} is in agreement with the results in \cite{dum1,kru2} where the canard point is obtained as a smooth function of $\sqrt{\epsilon}$.
\end{remark}


\section{Applications}
\label{sec:applications}
\subsection{Van der Pol}\seclab{vdp}
In this section we consider the classical van der Pol system
\begin{align*}
 \dot x &=F(x,y) = y-\frac{1}{3}x^3+x,\\
 \dot y&=G(x,z) = \epsilon (z-x).
 \end{align*}
 This system has a canard explosion near $z=\mu_0=1$ for $\epsilon$ small. Asymptotic expansions yield a more accurate value
 \begin{align}
  \mu=1-\frac{1}{8}\epsilon -\frac{3}{32}\epsilon^2-\frac{173}{1024}\epsilon^3 + \mathcal O(\epsilon^4),\eqlab{muexact}
 \end{align}
see e.g.\ \cite{ben1,zvonkin-1984:non-standard}. To use \algoref{algo} to compute the canard explosion point we first verify the conditions (A1) and (A2). Solving $F(x,y)$ for $y$ gives
\begin{align*}
 y=\zeta_0(x) = \frac{1}{3}x^3-x.
\end{align*}
Then
\begin{align*}
 \Lambda(x) = -\partial_x \zeta_0 \partial_y F(x,\zeta_0(x)) = 1-x^2.
\end{align*}
We have a fold point at $x=x_0=1$ where $\partial_x \zeta_0(x_0)=0$ where also $\Lambda$ vanishes. To complete the verification of (A1) we must solve $G(x_0,z)=0$ for $z=\mu_0$. We obtain $\mu_0=1$. For (A2) note that 
\begin{align}
 \tilde \Lambda(x) = -(1+x),\eqlab{tA0vdp}
\end{align}
dividing $x-x_0=x-1$ out, so that $\tilde \Lambda(x_0)=-2\ne 0$. Since $\Phi\equiv \epsilon$ we have to take case (b) in (A2). The remaining assumptions can easily be verified for $\epsilon$ sufficiently small.

We are now ready to apply \algoref{algo}. For $1^\circ$ we have \eqref{tA0vdp} and 
for $2^\circ$ we first note that
\begin{align*}
 \rho_0(x,z)  = \epsilon (z-x), 
\end{align*}
and so $\tilde e_0 = -\epsilon$ when dividing $x-x_0=x-1$ out from $e_0(x)=\rho_0(x,\mu_0)$.

For $3^\circ$ we first set $\zeta_1 = -\tilde e_0/\tilde \Lambda=-\epsilon /(1+x)$, which finishes step (i) and define $\rho_1$ (ii) as
\begin{align*}
 \rho_1(x,z) = \epsilon (1+z) -\epsilon (1+x)^{-3} (x^3 +3x^2+3x+1-\epsilon).
\end{align*}
Setting $\rho_1(x_0,z)=0$ gives $z=\mu_1=-\frac{1}{8}\epsilon$ so that 
\begin{align}
\mu^1=1+\mu_1 = 1-\frac{1}{8}\epsilon,\eqlab{vdpmu1}
\end{align}
correct to first order in $\epsilon$ cf. \eqref{muexact}. This finishes step (iii). We iterate this procedure and obtain the following approximations
\begin{align}
 \mu^2 = 1-\frac{1}{8}\epsilon + \frac{3}{32}\epsilon^2 - \frac{27}{2048}\epsilon^3,\,\mu^3 = 1-\frac{1}{8}\epsilon + \frac{3}{32}\epsilon^2 - \frac{173}{1024}\epsilon^3+\mathcal O(\epsilon^4),\eqlab{vdpmu2mu3}
\end{align}
correct to order $2$ respectively $3$ in $\epsilon$. 
\subsection{Templator}
In this section we consider the Templator model \cite{templatorref,bro2}
\begin{align}
 \dot x &=F(x,y)=k_u y^2+k_T y^2 x-\frac{qx}{K+x},\eqlab{templator}\\
 \dot y &=G(x,y,z) = z-k_u y^2 - k_T y^2 x,\nonumber
\end{align}
and use \algoref{algo} to compute a canard explosion. Numerical computations indicate a canard explosion at 
\begin{align}
\mu=0.419943,\eqlab{zTemplator}
\end{align}
\cite{templatorref,bro2}. This system has no explicit $\epsilon$, yet
numerical simulations show that the system exhibits a slow-fast
structure. In \cite{bro2} it is shown that various combinations of the
parameters in the system can locally be considered as small
parameters, but no global parametrization in the form
\eqref{xydot} exists. Here we proceed to find a canard explosion without any
identification of an explicit small parameter. As in \cite{bro2} we
set $k_u=0.01$, $k_T=1$, $q=1$, and $K=0.05$. We first solve
$F(x,y)=0$ for $y=\zeta_0(x)$ and obtain
\begin{align}
 \zeta_0(x) = \frac{50 \sqrt{2x(1+150x+5000 x^2)}}{1+150x+5000x^2}.\eqlab{zeta0Templator}
\end{align}
Note that \eqref{zeta0Templator} is independent of $z$. 
The equation also has a negative solution which we have discarded. We then realize that there is a point $x=\sqrt{2}/100=0.014142$ where $\partial_x \zeta_0 = 0$. The function $\Lambda(x)=-\partial_x \zeta_0 \partial_y F(x,\zeta_0)+\partial_y G(x,\zeta_0,z)$ vanishes near this point at $x=x_0=0.014345$. For this value of $x=x_0$, we continue to verify the assumptions in (A1), and compute $z=\mu_0$ giving $G(x_0,\zeta_0(x_0),z)=0$. We obtain $z=\mu_0=0.417681$. The error is $0.5\%$ in comparison with the value in \eqref{zTemplator}. We then define $\tilde \Lambda$ and $\tilde e_0$ by division of $\Lambda$ and $e_0$ by $x-x_0$. To verify (A2) we note that
\begin{align*}
 \Phi &\equiv  1,\\
\tilde \Lambda(x_0)&=996.78,\\
\tilde e_0(x_0) &=-17.157,\\
\end{align*}
so that
\begin{align}
\frac{\tilde e_0(x_0)}{\tilde \Lambda(x_0)} &= 0.017213.  \eqlab{small}
\end{align}
Since \eqref{small} is ``small'', we are confident that (A2), case (a), is satisfied and we therefore proceed by applying \algoref{algo}.
Introducing $\zeta_1(x) = -{\tilde e_0}(x)/{\tilde \Lambda}(x)$ then gives the new error $$\rho_{1}(x,z)=-\partial_x \zeta^1 F(x,\zeta^1)+G(x,\zeta^1,\mu_0+z),\quad \zeta^1=\zeta_0+\zeta_1.$$ We solve $\rho_1(x_0,z)=0$ for $z=\mu_1$ and obtain 
the improved approximation to the canard explosion point
\begin{align*}
 \mu^1 = \mu_0+\mu_1 = 0.419883,
\end{align*}
an error of $0.01 \%$. At the next step we get $\mu^2 =  0.419938$. The error is now $0.001 \%$. 
\subsection{Rotated van der Pol}
%
In this section we again consider the van der Pol equations, but we rotate the coordinates, replacing $x$ by $x-y$ and $y$ by $x+y$:
\begin{align*}
 \dot x &=F(x,y,z)=\frac{1}{2}(2x-(x-y)^3/3+\frac{1}{2}\epsilon(\mu-(x-y)),\\
 \dot y &=G(x,y,z) = -\frac{1}{2}(2x-(x-y)^3/3)+\frac{1}{2}\epsilon (\mu-(x-y)).
\end{align*}
mimicking a situation where the slow and fast variables have not been properly identified. We will demonstrate that the algorithm applies to this case too. 
The fold point in the original variables is no longer a fold point in the sense used in \secref{vdp} in the coordinates used here. Indeed solving for the $x$-nullcline gives
\begin{equation}
\label{eq:3}
 y= \zeta_0(x) = x-6^{1/3} x^{1/3}+\mathcal O(\epsilon),
\end{equation}
with non-zero derivative for the relevant $x$-values. The fold point
now appears where
$\Lambda(x) = 1-6^{2/3} x^{2/3}+ \epsilon 6^{-2/3} x^{-2/3}$
vanishes. The $\mathcal O(\epsilon)$-term in $\Lambda$ is following
the discussion in \remref{rema0} not important: It can collected into
$a_0=\mathcal O(\epsilon)$ in \eqref{xy0z0eqn}. We therefore re-define
$\Lambda$ as $\Lambda(x)=1-6^{2/3} x^{2/3}$. Then $\Lambda$ vanishes
at the point $x=1/6$, where $y=-5/6+\mathcal O(\epsilon)$ according to
(\ref{eq:3}). The point $(x,y)=(1/6,-5/6)$ is also the value obtained by transforming the fold point
point $(1,-2/3)$ in the original variables, used in \secref{vdp}, to
the current rotated variables. This alteration of $\Lambda$ is in
principle not needed: There is a point $\epsilon$-close to the point
above where the old $\Lambda$ vanishes. However, when we ignore this
part then the calculations can actually be done by hand without using
a computer algebra software.

Note that $\partial_x \Lambda(x_0) =-4\ne 0$ and we can therefore define
 \begin{align*}
  \tilde \Lambda(x) = -4+4(x-x_0)-\frac{32}{3}(x-x_0)^2 + \mathcal O((x-x_0)^3).
 \end{align*}
 Now we are in a position to apply \algoref{algo} $2^\circ$. For this we first note that
 \begin{align*}
  \rho_0(x,z) = -\partial_x \zeta_0 F(x,\zeta_0)+G(x,\zeta_0,z) = \epsilon 6^{-2/3} x^{-2/3} (z-6^{1/3} x^{1/3}),
 \end{align*}
and we therefore obtain $z=\mu_0=1$ by solving $\rho_0(x_0,z) = 0$. Following (iii), we then set
\begin{align*}
 \tilde e_0(x) = -\frac{2}{3}\epsilon (3-18(x-x_0)+104 (x-x_0)^2+\mathcal O((x-x_0)^3).
\end{align*}
It is easy to verify the conditions in (A2).
We therefore proceed as in \algoref{algo} ($3^\circ$) starting from (i) setting $\zeta_1(x) = -e_0(x)/\Lambda(x)=-\epsilon 6^{-2/3}/(6^{1/3} x^{1/3}+1)x^{2/3}$ which by construction is smooth also at $x=x_0$. We continue and obtain
\begin{align*}
\mu^1 &= \mu_0+\mu_1 = 1-\frac{1}{8}\epsilon -\frac{35}{32}\epsilon^2 - \frac{545}{384}\epsilon^3+\mathcal O(\epsilon^4),\\\mu^2 &= \mu_0+\mu_1+\mu_2 = 1-\frac{1}{8}\epsilon -\frac{3}{32}\epsilon^2 - \frac{18183}{2048}\epsilon^3+\mathcal O(\epsilon^4),\\
\mu^3 &=\mu_0+\mu_1+\mu_2+\mu_3= 1-\frac{1}{8}\epsilon + \frac{3}{32}\epsilon^2 - \frac{173}{1024}\epsilon^3+\mathcal O(\epsilon^4)
\end{align*}
correct to order $1$, $2$ respectively $3$ cf. \eqref{muexact}. Note that the corrections in $\mu^1$ (the terms $-\frac{35}{32}\epsilon^2 - \frac{545}{384}\epsilon^3$) and $\mu^2$ (the term $- \frac{18183}{2048}\epsilon^3$) are different from those in \eqsref{vdpmu1}{vdpmu2mu3} being $0$ and $- \frac{27}{2048}\epsilon^3$, respectively, but the expressions are correct to the order expected by \thmref{expest}, even though the slow and fast variables have not been properly identified.
 

\subsection{Templator again}\seclab{templator2}
As a final example we consider the Templator model \eqref{templator} again with the parameters $k_u=0.01$, $k_T=1$, $q=1$, and $K=0.05$ as before. The solution of $F(x,y)=0$ gave 
\begin{align*}
y= \zeta_0(x) = \frac{50 \sqrt{2x(1+150x+5000 x^2)}}{1+150x+5000x^2},
\end{align*}
having one single fold point at $x=x_0=\sqrt{2}/100$. There is,
however, another canard explosion at 
\begin{align}
\mu=0.967555,\eqlab{zTemplator2}
\end{align}
\cite{templatorref} that arises from another zero of $\Lambda$. Here
we cannot ignore the term $\partial_y G$ in
$\Lambda$ as we did in the rotated van der Pol above.
This term is not small as is illustrated in \figref{templator}. Here
it is made visible that $\Lambda$ has another zero at $x_0=0.599393$
where $\partial_x \Lambda< 0$. Therefore we can define
$\tilde \Lambda$ by $\Lambda(x)=(x-x_0)\tilde \Lambda(x)$. The zero of
$\Lambda$ gives rise to a singularity in the Eq. \eqref{update}. To continue the verification of (A1) we note that solving $\rho_0(x_0,z)=0$
gives
\begin{align*}
 \mu_0 = 0.967710.
\end{align*}
This gives a relative error of $0.01\%$ in comparison with the value in \eqref{zTemplator2}.
To verify the conditions (A2) we note the following:
\begin{align*}
\tilde \Lambda(x_0)&=-3.6535,\\
\tilde e_0(x_0) &=-0.0521311,
\end{align*}
and $\Phi \equiv 1$ as above,
so that
\begin{align}
 \frac{\tilde e_0(x_0)}{\tilde \Lambda(x_0)} &= 0.014268. \nonumber
\end{align}
This again gives us confidence to apply \algoref{algo}. The expressions are quite messy so we leave out the details and just present the result of one iteration of $2^\circ$ of the method: 
\begin{align*}
 \mu^1 = 0.967560.
\end{align*}
The error is now $0.006\%$. A final additional application gives \begin{align*}
 \mu^2 = 0.967558,
 \end{align*}
 reducing the relative error to $3\times 10^{-6}$.

\begin{figure}[h!]
\begin{center}
{\includegraphics[width=.5\textwidth]{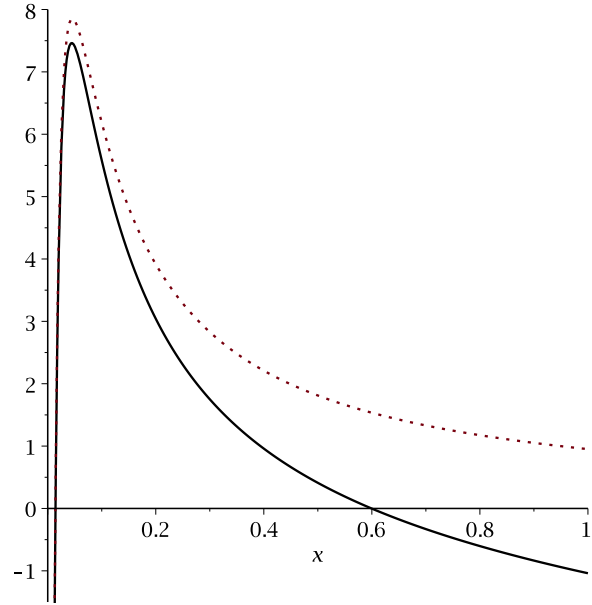}}
\end{center}
\caption{The graph of $\Lambda$ \eqref{A0eqn} and $\Lambda-\partial_y G=-\partial_x \zeta_0\partial_y F$ (dotted).}
\figlab{templator}
\end{figure}


\section{Proof of \thmref{expest}} \seclab{proof}

To prove our theorem we start from \eqref{xy0z0eqn}: 
\begin{align}
 \dot x & =F_0(x,y_0,z_0)=f_0(x)+F_{0y}(x,y_0,z_0) y_0+F_{0z}(x,y_0,z_0)z_0,\eqlab{xy0z0eqn1}\\
 \dot y_0&= e_0(x)+(\Phi(x)+b_0(x))z_0+T_0(x,z_0)+(\Lambda(x)+a_0(x,z_0))y_0+R_0(x,y_0,z_0),\nonumber\\
 \dot z_0&=0,\nonumber
\end{align}
where $f_0\equiv 0$ and $b_0\equiv 0$. 
From conditions (A1) and (A2), it is without loss of generality to take $K\equiv \vert \tilde \Lambda^{-1}\vert_{\nu_0}=1$. Indeed, we can introduce a new time $\tau=K^{-1}t$ to achieve this. Then also by (A2):
\begin{align}
 \tilde \delta_0 \equiv \vert \tilde e\vert_{\nu_0} \le \epsilon\ll 1,\eqlab{delta0}
 \end{align}
and furthermore, assuming case (a):
\begin{align}
 \vert \Phi(x_0)\vert ^{-1},\, \vert \Phi\vert \le C_\Phi. \eqlab{B0est}
\end{align}
Here $C_\Phi>0$ is independent of $\epsilon$. 
The proof for case (b) where 
\begin{align}
 \vert \Phi(x_0)\vert ^{-1}\le C_\Phi/\epsilon,\, \vert \Phi\vert \le C_\Phi\epsilon, \eqlab{B0est2bNew}
\end{align}
is almost identical (see \remref{caseb} below).

We will seek to apply a sequence of transformations $\phi_1$, $\phi_2$, $\ldots$, $\phi_n$ to \eqref{xy0z0eqn} that successively seek to diminish the $\tilde e_i$'s that appear as a result of these transformations. We define the transformations through an iterative lemma.

\textbf{The iterative lemma}. To set up an iterative lemma we start from the normal form \eqref{xy0z0eqn1} with appropriate subscripts removed:
\begin{align}
 \dot x & =F(x,y,z)= f(x) +F_y(x,y,z)y+F_z(x,y,z)z,\eqlab{xyzeqn}\\
 \dot y&= e(x)+(\Phi(x)+b(x))z+T(x,z)+(\Lambda(x)+ a(x,z))y+ R(x,y,z),\nonumber\\
 \dot z&=0,\nonumber
\end{align}
where $$e=(x-x_0)\tilde e(x),$$ 
$T=\mathcal O(z^2)$ independent of $y$, and $R=\mathcal O(y^2)$. We suppose
\begin{align}
\tilde \delta = \vert \tilde e\vert_{\nu},\, \vert a(x,0)\vert_\nu,\,\vert \epsilon^{-1} b(x_0)\vert &\le c\epsilon,\eqlab{best}
\end{align}
for some constant $c>0$ and $\epsilon$ sufficiently small. Note that these conditions are satisfied for our initial system \eqref{xy0z0eqn1}.

Then $y=0,\,z=0$ is an approximation to the canard solution. The accuracy of the approximation is determined by $\tilde \delta$. 
We will in the following apply a transformation $\phi$ to \eqref{xyzeqn}, which will be based on two steps, seeking to improve the approximation of the canard solution. First, we define the solution, $\zeta=\zeta(x)$, of the equation
\begin{align}
e(x)&+\Lambda(x) \zeta(x)=0.\eqlab{zeta}
\end{align}
The solution is
\begin{align}
 \zeta(x) = \frac{\tilde e(x)}{\tilde \Lambda(x)},\eqlab{zetasol}
\end{align}
measuring
\begin{align}
 \vert \zeta\vert_\nu \le {K\tilde \delta}.\eqlab{zetaest}
\end{align}
Now, we set
\begin{align*}
 y=\zeta(x)+y_+.
\end{align*}
This gives
\begin{align*}
 \dot x &=  f^+(x)+F_y^+(x,y_+,z)y_++F_z^+(x,y_+,z)z,\\
 \dot y_+&= e^+(x)+(\Phi(x)+b^+(x))z+T^+(x,z)+(\Lambda(x)+a^+(x,z))y+R^+(x,y,z),\nonumber\\
 \dot z&=0,\nonumber
\end{align*}
with
\begin{align*}
  f^+(x) &=f(x)+F_{y}(x,\zeta,0)\zeta,\\
 F_{z}^+(x,y_+,z) &=F_z(x,\zeta+y_+,z),\\
 e^+(x)&=- \partial_x\zeta   f^+(x)+a(x,0)\zeta + R_+(x,\zeta,0),\\
 b^+(x)&=b(x)-\partial_x\zeta F_{z}(x,\zeta,0)+\partial_z a(x,0)\zeta + \partial_z R_+(x,\zeta,0),\\
 a^+(x,z)&=a(x,z)-\partial_x \zeta F_{y}(x,\zeta,0)+\partial_y R_+(x,\zeta,0),
\end{align*}
and where $F_{y}^+$, $T^+=\mathcal O(z^2)$, which is independent of $y$, and $R^+=\mathcal O(y^2)$ are determined by Taylor's theorem. For $\epsilon$ sufficiently small we then have
\begin{align*}
 \delta^+\equiv \vert e^+\vert_{\nu-\xi} \le \frac{\epsilon C \tilde \delta }{\xi},
\end{align*}
for $C$ sufficiently large, using \eqref{best}, \eqref{zetaest} and the fact that $R^+(x,y,z)=\mathcal O(y^2)$ is quadratic.

The result is not yet appropriate for iteration as
\begin{align*}
 e^+(x_0) \ne 0.
\end{align*}
To account for this we transform $z$ by introducing $z=\mu+z_+$ with $\mu$ satisfying
\begin{align}
 e^+(x_0)+(\Phi(x_0)+b^+(x_0))\mu +T^+(x_0,\mu)= 0.\eqlab{musol}
\end{align}
By \eqref{B0est}, and the contraction mapping theorem, there exists a solution 
\begin{align*}
 \mu \approx -\frac{e^+(x_0)}{\Phi(x_0)},
\end{align*}
of \eqref{musol}, that satisfies
\begin{align}
 \vert \mu\vert \le 2C_\Phi\delta^+\le \frac{2\epsilon C_\Phi C\tilde \delta}{\xi}.\eqlab{muest}
\end{align}
Here we have used that $T^+=\mathcal O(z^2)$, \eqref{best}, the first estimate in \eqref{B0est} and the smallness of $\tilde \delta$ and $\epsilon$. Then the resulting system reads
\begin{align}
 \dot x & =F(x,\zeta+y_+,\mu+z_+)= f_+(x)+F_{+y}(x,y_+,z_+)y_++F_{+z}(x,y_+,z_+)z_+,\eqlab{xyzeqnplus}\\
 \dot y_+&= e_+(x)+(\Phi(x)+b_+(x))z_++T_+(x,z_+)+(\Lambda(x)+a_+(x,z_+))y+R_+(x,y,z_+),\nonumber\\
 \dot z_+&=0,\nonumber
\end{align}
with
\begin{align*}
  f_+(x) &=f^+(x)+F_z^+(x,0,\mu)\mu,\\
  F_{+y}(x,y_+,z_+) &= F^+_y(x,y_+,\mu+z_+)\\
 e_+(x)&=e^+(x)+(\Phi(x)+b^+(x))\mu +T^+(x,\mu),\\
 b_+(x)&=b^+(x)+\partial_z T^+(x,\mu),
\end{align*}
and where $F_{+z}$, $T_+$, $R_+$ are determined by Taylor's theorem. By construction $e_+(x_0)=0$. We estimate the new obstacle to invariance $e_+$ of $y_+=0,\,z_+=0$ as
\begin{align}
 \vert e_+ \vert_{\nu-\xi} &\le \vert e^+\vert_{\nu-\xi} +2 C_{\Phi} \vert \mu\vert\nonumber\\
 &\le \frac{\epsilon C(1+2C_{\Phi}^2)\tilde \delta }{\xi},\eqlab{eplus}
\end{align}
using the second estimate in \eqref{B0est} and \eqref{best}, 
for $\tilde \delta$ and $\epsilon$ sufficiently small. Since $e_+(x_0) = 0$ we write $e_+$ as
\begin{align*}
 e_+(x) = (x-x_0)\tilde e_+(x),\quad \tilde e_+(x) = \int_0^1 \partial_x e_+(x_0+s(x-x_0)) ds,
\end{align*}
and estimate our new error $\tilde e_+$ on $\nu_+=\nu-2\xi$ using a Cauchy estimate and \eqref{eplus}:
\begin{align*}
 \vert \tilde e_+\vert_{\nu_+} \le \frac{\epsilon C(1+2C_{\Phi}^2)\tilde \delta }{\xi^2}.
\end{align*}
\begin{remark}\remlab{caseb}
 If we assume case (b) in (A2) then we obtain a $\epsilon^{-1}$-factor in \eqref{muest} so that $\vert \mu\vert \le 2\epsilon^{-1} \vert \Phi(x_0)\vert^{-1} \delta^+$. However, we recover an estimate as in \eqref{eplus} using \eqref{B0est2bNew} (in place of \eqref{B0est} used above). 
\end{remark}

We collect the results in the following iterative lemma:
\begin{lemma}\lemmalab{itlam}
 Let $\xi>0$ and $\nu_+\equiv \nu-2\xi\ge 0$. Then there exists a $\bar C>0$ so that the transformation
 \begin{align*}
\phi: \,(x,y,z)  &\mapsto (x_+,y_+,z_+),\\
x_+ &= x,\\
y_+ &=\zeta(x)-y,\\
z_+&=\mu-z,
 \end{align*}
 with $\zeta$ and $\mu$ solving \eqsref{zetasol}{musol}, respectively, 
maps \eqref{xyzeqn} into \eqref{xyzeqnplus} where
\begin{align*}
 \vert a-a_+\vert_{\nu_+},\,\vert b-b_+\vert_{\nu_+},\,\vert R-R_+\vert_{\nu_+},\,\vert T-T_+\vert_{\nu_+} \le \bar C \tilde \delta,
\end{align*}
and
\begin{align}
 \tilde \delta_+\equiv \vert \tilde e_+\vert_{\nu_+} \le\frac{\epsilon \bar C}{\xi^2}\tilde \delta,\eqlab{tdeltaplus}
\end{align}
provided $\epsilon$ are sufficiently small. 
\end{lemma}

The $\mathcal O(\epsilon^n)$ estimates in (i) of \thmref{expest} follow directly from \eqref{tdeltaplus} as each application of the procedure introduces a factor of $\epsilon$.

\textbf{Exponential estimates}. To obtain the exponential estimates we first apply the iterative lemma, \lemmaref{itlam}, to \eqref{xy0z0eqn1} and obtain 
\begin{align}
 \tilde \delta_1= \vert \tilde e_1\vert_{\nu_1} \le \frac{\epsilon \bar C_0}{\xi_0^2}\tilde \delta_0 = \frac{16 \bar C_0}{\nu_0^2}\epsilon^2,\eqlab{tdelta1}
\end{align}
using \eqref{delta0} and \eqref{tdeltaplus} with $\nu_1 = \frac{\nu_0}{2}$ setting here $\xi=\xi_0=\frac{\nu_0}{4}=\mathcal O(1)$.\footnote{We could in principle take any $\xi=\mathcal O(1)$ satisfying $\xi<\frac{\nu_0}{2}$} Here $C_0$ is the constant obtained from applying \lemmaref{itlam} to \eqref{xy0z0eqn1}. Then we apply the lemma successively setting the measure of the domain reduction $\xi$ in \lemmaref{itlam} to be\footnote{In fact, any $\xi = C^{-1}\sqrt{\epsilon}$ for $C$ sufficiently large would do.}  $$\xi=\sqrt{2\epsilon {\bar C}_\infty},$$ with 
\begin{align}
\bar {C}_\infty=2\bar C_{0}.\eqlab{Cbar0}
\end{align}
The $(n+1)$th application of \lemmaref{itlam} gives rise to a constant $\bar C_{n}$. Then by \eqref{tdeltaplus}
\begin{align}
 \tilde \delta_{n+1} \le  \frac{\epsilon\bar C_{n+1}}{\xi_n^2}\tilde \delta_{n}\le 
 2^{-1} \tilde \delta_{n-1}\le 2^{-n} \tilde \delta_1,\eqlab{tdeltan}
\end{align}
while $\bar C_{n+1}\le {\bar C}_\infty= 2\bar C_0$ and
\begin{align}
 \nu_{n+1} = \nu_{n} -2\xi = \frac12 \nu_{0}-2n\xi\ge 0.\eqlab{nun}
\end{align}
By the geometric sum formula and the fact that $\tilde \delta_1=\mathcal O(\epsilon^2)$, cf. \eqref{tdelta1}, the requirement $\bar C_{{n}}\le 2\bar C_{0}$ does not pose any restrictions on $n$ for $\epsilon$ sufficiently small. 
The only restriction on $n$ is contained in the last inequality in \eqref{nun}:
\begin{align*}
 n \le N(\epsilon) \equiv \left \lfloor{\frac{\nu_0}{4\xi}}\right \rfloor =\left \lfloor\frac{\nu_0}{4\sqrt{2\epsilon {\bar C}_\infty}}\right \rfloor.
\end{align*}
Here $\lfloor v \rfloor$ denotes the integer part of $v\ge 0$. Hence \lemmaref{itlam} can be applied $\mathcal O(\epsilon^{1/2})$-many times, which by \eqref{tdeltan} with $n=N(\epsilon)$ gives the exponential estimate
\begin{align*}
 \tilde \delta_{N(\epsilon)+1} \le 2^{-\left \lfloor\frac{\nu_0}{4\sqrt{2\epsilon {\bar C}_\infty}}\right \rfloor} \tilde \delta_1,
\end{align*}
in \thmref{expest}. This then completes the proof.
\bibliography{finalbib}
\bibliographystyle{plain}

\end{document}